\documentclass[12pt]{amsart}
\usepackage{graphicx}
\usepackage{amsmath,amssymb}
\newtheorem{thm}{Theorem}[section]
\newtheorem{cor}[thm]{Corollary}
\newtheorem{lem}[thm]{Lemma}

\theoremstyle{definition}

\theoremstyle{remark}

\numberwithin{equation}{section}

\begin{document}
\title[Hyers-Ulam stability ]{Hyers-Ulam stability of
Jensen functional  equation on amenable semigroups}
\author[ B. Bouikhalene and  E. Elqorachi  ]{Bouikhalene Belaid and Elqorachi Elhoucien  }


\begin{abstract}
In this paper, we give a  proof of the  Hyers-Ulam stability of
the Jensen functional equation
$$f(xy)+f(x\sigma(y))=2f(x),\phantom{+} x,y\in{G},$$ where $G$ is an amenable semigroup and $\sigma$ is an involution of $G.$
\end{abstract}
\maketitle
\section{Introduction}
In 1940,  Ulam \cite{43} gave a wide ranging talk before the
Mathematics Club of the University of wisconsin in which he
discussed a number of important unsolved problems. Among those was
the question concerning the stability of group homomorphisms:
Given a group $G_{1}$, a metric group $(G_{2},d)$, a number
$\epsilon>0$ and a mapping $f$: $G_{1}\longrightarrow G_{2}$ which
satisfies $d(f(xy),f(x)f(y))<\epsilon$ for all $x,y\in{G_{1}}$,
does there exist an homomorphism $g$: $G_{1}\longrightarrow G_{2}$
and a constant $k>0$, depending only on $G_{1}$ and $G_{2}$ such
that $d(f(x),g(x))<k\epsilon$ for all $x\in{G_{1}}$?\\ In 1941,
Hyers \cite{H} considered the case of approximately additive
mappings under the assumption that $G_{1}$ and $G_{2}$ are Banach
spaces.\\    Rassias \cite{19} provided a generalization of the
Hyers' Theorem for linear mappings, by allowing the Cauchy
difference to be unbounded.
\\Beginning around the year 1980, several results for the
Hyers-Ulam-Rassias stability of very many functional equations
have been proved by several researchers.
For more detailed, we can refer for example  to \cite{4c}, \cite{Ga}, \cite{Gav}, \cite{H2}, \cite{Jung}, \cite{12f}, \cite{Yang}. \\
Let $G$ be a semigroup with neutral element $e$. Let $\sigma$ be
an involution of $G$, which means that
$\sigma(xy)=\sigma(y)\sigma(x)$ and $\sigma(\sigma(x))=x$ for all
$x,y\in G$.\\We say that $f$: $G\longrightarrow \mathbb{C}$
satisfies the Jensen functional equation if
\begin{equation}\label{eq1}
    f(xy)+f(x\sigma(y))=2f(x)
\end{equation} for all $x,y\in G$. \\The Jensen functional equation (\ref{eq1}) takes the form \begin{equation}\label{eq2}
    f(xy)+f(xy^{-1})=2f(x)
\end{equation} for all $x,y\in G$, when $\sigma(x)=x^{-1}$ and $G$
is a group.
\\The stability in the sens of Hyers-Ulam of the Jensen equations (\ref{eq1}) and (\ref{eq2})  has been studied
by various authors when $G$ is an abelian group or a vector space.
 The interested reader can be referred  to   Jung \cite{3},
  Kim \cite{4} and   Bouikhalene et al. \cite{a,b}. \\
 Recently,  Faiziev and  Sahoo \cite{1} have proved the
 Hyers-Ulam stability of equation (\ref{eq2}) on some noncommutative groups such as metabelian groups and $T(2,K)$, where
 $K$ is an arbitrary commutative field with characteristic different from two.\\ In this paper, motivated by the ideas of
  Forti,  Sikorska \cite{2} and  Yang \cite{5}, we give a
 proof of the Hyers-Ulam stability of the Jensen functional
 equation (\ref{eq1}), under the condition that $G$ is an amenable
 semigroup.
\section{Hyers-Ulam stability of  equation (\ref{eq1})
on amenable semigroups} In this section we investigate the
Hyers-Ulam stability of equation (\ref{eq1}) on an amenable
semigroup $G$. \\We recall that a semigroupe $G$ is said to be
amenable if there exists an invariant mean on the space of the
bounded complex functions defined on $G$. We refer to \cite{7} for
the definition and properties of invariant means.\\The main result
of the present paper is the following.
\begin{thm}Let $G$ be an amenable semigroup with neutral element $e$.
Let $f$: $G\longrightarrow \mathbb{C}$ be a function satisfying the following inequality
\begin{equation}\label{eq21}
\mid f(xy)+f(x\sigma(y))-2f(x)\mid\leq \delta
\end{equation} for all $x,y\in G$ and for some nonnegative $\delta$.
Then there exists a unique solution $g$ of the Jensen equation (\ref{eq1}) such that
\begin{equation}\label{eq22}
    \mid f(x)-g(x)-f(e)\mid\leq 3\delta
\end{equation}for all $x\in G.$\end{thm}
First, we recall the following lemma which is a generalization of
the useful lemma obtained by  Forti and  Sikorska in \cite{2}. The
proof was given by the authors in \cite{E}.
\begin{lem}\label{lem21} Let $G$ be a semigroup and $B$ be
a Banach space. Let $f:G\rightarrow B$ be a function for which
there exists a solution $g$ of the Drygas functional equation
\begin{equation}\label{eq23}g(yx)+g(\sigma(y)x)=2g(x)+g(y)+g(\sigma(y)),\;\;\;x,y\in
G
\end{equation} such that  $\|f(x)-g(x)\|\leq M$, for all
$x\in{G}$ and for some $M\geq0$. Then
\begin{equation}\label{eq24}
g(x)=\lim_{n\rightarrow+\infty}2^{-2n}\{f^{e}(x^{2^{n}})+\frac{1}{2}\sum^{n}_{k=1}2^{k-1}
[f^{e}((x^{2^{n-k}}\sigma(x)^{2^{n-k}})^{2^{k-1}})
+f^{e}((\sigma(x)^{2^{n-k}}x^{2^{n-k}})^{2^{k-1}})]\}\end{equation}$$+2^{-n}\{f^{o}(x^{2^{n}})
 +\frac{1}{2}\sum^{n}_{k=1}
[f^{e}((x^{2^{k-1}}\sigma(x)^{2^{k-1}})^{2^{n-k}})-f^{e}((\sigma(x)^{2^{k-1}}x^{2^{k-1}})^{2^{n-k}})]\},$$
where
$f^{e}(x)=\frac{f(x)+f(\sigma(x))}{2},f^{o}(x)=\frac{f(x)-f(\sigma(x))}{2}$
are the even and odd part of $f$.
\end{lem}
\begin{lem} Let $G$ be semigroup with neutral element $e$ and $B$ a complex Banach space.
Assume that $f$: $G\longrightarrow B$ satisfies the inequality
\begin{equation}\label{eq25} \mid
f(xy)+f(x\sigma(y))-2f(x)\mid\leq \delta
\end{equation} for all $x,y\in G$ and for  some  $\delta\geq 0$.
Then the limit \begin{equation}\label{eq26}
    g(x)=\lim_{n\longrightarrow \infty}2^{-n}f(x^{2^{n}})
\end{equation}   exists for all $x\in G$  and satisfies
\begin{equation}\label{eq27}
    \mid f(x)-g(x)-f(e)\mid\leq\frac{3\delta}{2}\;\;\text{and}\;\;
    g(x^{2})=2g(x)
\end{equation}for all $x\in G.$\\The function $g$ with the
condition (\ref{eq27}) is unique.
\end{lem}\begin{proof}We define a function $h$: $G\longrightarrow \mathbb{C}$ by $h(x)=f(x)-f(e)$. First, by using
the inequality (\ref{eq25})  we obtain
\begin{equation}\label{eq28} \mid
h(xy)+h(x\sigma(y))-2h(x)\mid\leq \delta
\end{equation} for all $x,y\in G$. Letting $x=e$ in (\ref{eq28}), we get \begin{equation}\label{eq29}
    \mid h^{e}(y)\mid\leq\frac{\delta}{2}
\end{equation} for all $y\in G.$ Similarly, we can put $y=x$ in (\ref{eq28}) to obtain
\begin{equation}\label{eq210}
    \mid h(x^{2})+h(x\sigma(x))-2h(x)\mid\leq\delta
\end{equation} for all $x\in G$. Since $h(x\sigma(x))=h^{e}(x\sigma(x))$, so from (\ref{eq29}) and  (\ref{eq210}), we have
\begin{equation}\label{211}
    \mid h(x^{2})-2h(x)\mid\leq\frac{3\delta}{2}
\end{equation}for all $x\in G.$ Now, by applying some approach used  in \cite{1}, we get the rest of the proof.
\end{proof}
\textbf{Proof of Theorem 2.1.} In the proof we use some ideas from
Yang \cite{5} and Forti, Sikorska \cite{2}.\\ Setting $x=e$ in
(\ref{eq21})   we have
\begin{equation}\label{eq212}
    \mid f^{e}(y)-f(e)\mid\leq\frac{\delta}{2}
\end{equation}for all $y\in G.$ \\ The inequalities (\ref{eq21}), (\ref{eq212}) and the triangle inequality gives
\begin{equation}\label{eq213}
    \mid f(xy)+f(yx)-2f(x)-2f(y)+2f(e)\mid\end{equation}$$\leq\mid
    f(xy)+f(x\sigma(y))-2f(x)\mid+\mid f(yx)+f(y\sigma(x))-2f(y)\mid$$$$+\mid
    2f(e)-f(x\sigma(y))-f(y\sigma(x))\mid$$$$\leq3\delta.
$$Hence, from (\ref{eq21}) (\ref{eq212}) and (\ref{eq213}) we get \begin{equation}\label{eq214}
    \mid f(yx)+f(\sigma(y)x)-2f(x)\mid\leq\mid
    f(yx)+f(xy)-2f(y)-2f(x)+2f(e)\mid\end{equation}$$+\mid
    f(\sigma(y)x)+f(x\sigma(y))-2f(\sigma(y))-2f(x)+2f(e)\mid$$
$$+\mid -f(xy)-f(x\sigma(y))+2f(x)\mid+\mid
2f(y)+2f(\sigma(y))-4f(e)\mid\leq9\delta$$   Now, from
(\ref{eq21}) and (\ref{eq214}) we obtain
\begin{equation}\label{eq215}
\mid
f(yx)-f(\sigma(x)\sigma(y))+f(y\sigma(x))-f(x\sigma(y))-2(f(y)-f(\sigma(y)))\mid
\end{equation}
$$\leq \mid f(yx)+f(y\sigma(x))-2f(y)\mid +\mid f(x\sigma(y))+f(\sigma(x)\sigma(y))-2f(\sigma(y))\mid$$
$$\leq 10\delta.$$
Consequently, we get
\begin{equation}\label{eq216}
\mid f^{o}(yx)+f^{o}(y\sigma(x))-2f^{o}(y)\mid\leq5\delta
\end{equation}for all $x,y\in G.$
So, for fixed $y\in G$, the functions $x\longmapsto
f^{o}(yx)-f^{o}(x\sigma(y))$ and $x\longmapsto
f^{o}(xy)+f^{o}(x\sigma(y))-2f^{o}(x)$ are bounded on $G$.
Furthermore, \begin{equation}\label{eq217}
m\{f^{o}_{\sigma(y)\sigma(z)}+f^{o}_{\sigma(y)z}-2f^{o}_{\sigma(y)}\}=m\{(f^{o}_{\sigma(z)}+f^{o}_{z}-2f^{o})_{\sigma(y)}\}
\end{equation}
$$=m\{f^{o}_{\sigma(z)}+f^{o}_{z}-2f^{o}\},$$ where $m$ is an invariant mean on $G$.\\
 By using some computation  as the one of  (\ref{eq214})
we get,  for every fixed $y\in G$ the function $x\longmapsto
f^{o}(yx)+f^{o}(\sigma(y)x)-2f^{o}(x)$ is bounded and
\begin{equation}\label{eq218}
m\{_{zy}f^{o}+_{\sigma(z)y}f^{o}-2_{y}f^{o}\}=m\{_{y}(_{z}f^{o}+_{\sigma(z)}f^{o}-2f^{o})\}
\end{equation}
$$=m\{_{z}f^{o}+_{\sigma(z)}f^{o}-2f^{o}\}$$Now, define \begin{equation}\label{eq219}
    \phi(y):=m\{_{y}f^{o}-f^{o}_{\sigma(y)}\},\;\;\;y\in G.
\end{equation} By using the definition of $\phi$ and $m$ the equalities (\ref{eq217}) and
(\ref{eq218}), we get \begin{equation}\label{eq220}
\phi(zy)+\phi(\sigma(z)y)=m\{_{zy}f^{o}-f^{o}_{\sigma(y)\sigma(z)}\}+m\{_{\sigma(z)y}f^{o}-f^{o}_{\sigma(y)z}\}
\end{equation}
$$=m\{_{zy}f^{o}+_{\sigma(z)y}f^{o}-2_{y}f^{o}\}-m\{f_{\sigma(y)\sigma(z)}^{o}+f_{\sigma(y)z}^{o}-2f_{\sigma(y)}^{o}\}$$
$$+2m\{_{y}f^{o}-f^{o}_{\sigma(y)}\}$$
$$=m\{_{z}f^{o}+_{\sigma(z)}f^{o}-2f^{o}\}-m\{f_{\sigma(z)}^{o}+f_{z}^{o}-2f^{o}\}$$
$$+2m\{_{y}f^{o}-f^{o}_{\sigma(y)}\}$$
$$=m\{_{z}f^{o}-f^{o}_{\sigma(z)}\}+m\{_{\sigma(z)}f^{o}-f^{o}_{z}\}+2m\{_{y}f^{o}-f^{o}_{\sigma(y)}\}$$
$$=2\phi(y)+\phi(z)+\phi(\sigma(z)),$$
which implies that $\phi$ is a solution of the Drygas functional
equation (\ref{eq23}). Furthermore, we have
\begin{equation}\label{221}
    \mid
    \frac{\phi}{2}(y)-f^{o}(y)\mid=\frac{1}{2}\mid\phi(y)-2f^{o}(y)\mid=\frac{1}{2}\mid
    m\{_{y}f^{o}-f^{o}_{\sigma(y)}-2f^{o}(y)\}\mid
\end{equation}$$\leq\frac{1}{2}\mid m\mid \text{sup}_{x\in G}\mid f^{o}(yx)-f^{o}(x\sigma(y))-2f^{o}(y)\mid=\frac{1}{2} \text{sup}_{x\in G}\mid
f^{o}(yx)+f^{o}(y\sigma(x))-2f^{o}(y)\mid$$$$\leq\frac{5}{2}\delta.$$
Now, by using Lemma 2.2: the mapping $\frac{\phi}{2}$ is a
solution of Drygas functional equation (\ref{eq23}) and
$\frac{\phi}{2}-f^{o}$ is a bounded mapping, so we have
\begin{equation}\label{}
    \frac{\phi}{2}=\lim_{n\rightarrow+\infty}2^{-n}
    f^{o}(x^{2^{n}}),
\end{equation}which implies that $\frac{\phi}{2}$ is odd, so
$\frac{\phi}{2}$ is a solution of Jensen functional equation
(\ref{eq1}). Consequently, we have
\begin{equation}\label{}
    \mid f(x)-\frac{\phi}{2}-f(e)\mid=\mid f^{e}(x)+f^{o}(x)-\frac{\phi}{2}-f(e)\mid\end{equation}$$\leq\mid
    f^{e}(x)-f(e)\mid+\mid f^{o}(x)-\frac{\phi}{2} \mid$$
    $$\leq\frac{\delta}{2}+\frac{5\delta}{2}=\frac{\delta}{2}=3\delta.$$
For proving the uniqueness of the obtained solution, we use the
following: if $g$ is a solution of equation (\ref{eq1}), then
\begin{equation}\label{eqj}
g(e)=g^{e}(x)=g(x\sigma(x));\;\;g(x^{2^{n}})+(2^{n}-1)g(e)=2^{n}g(x)
\end{equation}for all $n\in \mathbb{N}$ and for all $x\in G.$\\
By using (\ref{eqj}) and the proof of [Propostion 3,  \cite{5}] we
get the following result.
\begin{cor} Let $G$ be an amenable semigroup with neutral
element $e$ and $B$ a Banach space Let $f$: $G\longrightarrow B$
be a function satisfying the following inequality
\begin{equation}\label{eq222}
\parallel f(xy)+f(x\sigma(y))-2f(x)\parallel\leq \delta
\end{equation} for all $x,y\in G$ and for some nonnegative $\delta$.
Then there exists a unique solution $g$ of the Jensen equation
(\ref{eq1}) such that
\begin{equation}\label{eq223}
    \parallel f(x)-g(x)-f(e)\parallel\leq3\delta
\end{equation}for all $x\in G.$\end{cor}
\begin{cor} Let $G$ be an amenable semigroup with neutral element
$e$ and $B$ a Banach space Let $f$: $G\longrightarrow B$ be a
function satisfying the following inequality
\begin{equation}\label{eq224}
\parallel f(xy)+f(xy^{-1})-2f(x)\parallel\leq \delta
\end{equation} for all $x,y\in G$ and for some nonnegative $\delta$.
Then there exists a unique solution $g$ of the Jensen equation
(\ref{eq2}) such that
\begin{equation}\label{eq225}
    \parallel f(x)-g(x)-f(e)\parallel\leq3\delta
\end{equation}for all $x\in G.$\end{cor}

\vspace{1cm} Belaid Bouikhalene,
\\Department of Mathematics, \\Polydisciplinary Faculty, Sultan Moulay Slimane University,
\\Beni-Mellal, Morocco, E-mail: bbouikhalene@yahoo.fr\\ \\Elqorachi
Elhoucien,   \\Department of Mathematics, Faculty of Sciences,
Ibn Zohr University, \\Agadir, Morocco, E-mail:
elqorachi@hotmail.com

\end{document}